\documentclass[journal]{IEEEtran}
\usepackage{xcolor}
\usepackage{cite}
\usepackage{amssymb}
\usepackage{amsmath}
\usepackage{amsthm}
\usepackage{lineno}
\usepackage{graphicx}
\usepackage{verbatim}
\usepackage{bm}
\usepackage{mathtools}
\usepackage{caption}
\usepackage{subcaption}

\newtheorem{theorem}{Theorem}
\newtheorem{proposition}{Proposition}
\newtheorem{lemma}{Lemma}

\newtheorem{fact}{Fact}

\newcommand{\R}{\mathbb R}
\newcommand{\limto}{\rightarrow}

\IEEEoverridecommandlockouts                              

\title{
Local Strong Convexity of Source Localization and Error Bound for Target Tracking under Time-of-Arrival Measurements
}

\author{Yuen-Man Pun and Anthony Man-Cho So,~\IEEEmembership{Senior Member,~IEEE} 
\thanks{Y.-M. Pun and A. M.-C. So are with Department of Systems Engineering and Engineering Management, The Chinese University of Hong Kong, Hong Kong. (email: ympun@se.cuhk.edu.hk; manchoso@se.cuhk.edu.hk)}
\thanks{A preliminary version of this paper~\cite{PS20} has appeared in the Proceedings of the 59th IEEE Conference on Decision and Control (CDC 2020). This work is supported in part by CUHK Research Sustainability of Major RGC Funding Schemes project 3133236.}
}

\begin{document}

\maketitle
\thispagestyle{empty}
\pagestyle{empty}

\begin{abstract}
In this paper, we consider a time-varying optimization approach to the problem of tracking a moving target using noisy time-of-arrival (TOA) measurements. Specifically, we formulate the problem as that of sequential TOA-based source localization and apply online gradient descent (OGD) to it to generate the position estimates of the target. To analyze the tracking performance of OGD, we first revisit the classic least-squares formulation of the (static) TOA-based source localization problem and elucidate its estimation and geometric properties. In particular, under standard assumptions on the TOA measurement model, we establish a bound on the distance between an optimal solution to the least-squares formulation and the true target position. Using this bound, we show that the loss function in the formulation, albeit non-convex in general, is locally strongly convex at its global minima. To the best of our knowledge, these results are new and can be of independent interest. By combining them with existing techniques from online strongly convex optimization, we then establish the first non-trivial bound on the cumulative target tracking error of OGD. Our numerical results corroborate the theoretical findings and show that OGD can effectively track the target at different noise levels.
\end{abstract}


\section{Introduction} \label{sec:intro}
Target tracking~\cite{LWH+02,SNP16} is a key enabling technology in many applications of multi-agent systems and wireless sensor networks, such as motion planning~\cite{CWC+04,DSH09} and surveillance~\cite{PJ05,PDS14}. In one of its basic forms, the tracking problem aims to maintain position estimates of a moving, signal-emitting target over time using noisy measurements of the emitted signal collected by stationary sensors. Such a sequential localization formulation has been extensively studied in the control and signal processing communities, and various approaches for tackling it have been proposed. When a model on the target dynamics and noise statistics is available, a classic approach is to employ Kalman filtering techniques to perform the tracking; see, e.g.,~\cite{LEH06,TFLC09,WLM11,YRLS14} and the references therein. In recent years, however, there have been increasing efforts in developing tracking techniques that require only minimal assumptions on the target trajectory and/or noise distribution. One approach is to view the sequential localization formulation through the lens of time-varying optimization~\cite{DSBM20,SDP+20}. Specifically, at each time step, the position estimate of the target is given by a minimizer of a loss function that depends on the noisy signal measurements collected at that time step. However, since the time interval between successive measurements is often very short and the sensors have limited computational power, it is impractical to solve the loss minimization problem at each time step exactly. This motivates the use of online optimization techniques to tackle the target tracking problem. To evaluate the performance of an online method, various metrics are available; see~\cite{SDP+20}. These metrics differ in how they measure the discrepancy between the solutions generated by the method at different time steps and the optimal solutions at the corresponding time steps. When the loss function is convex at every time step, it has been shown that many online methods enjoy strong performance guarantees under different metrics; see, e.g.,~\cite{Zink03,HW15,JRSS15,MSJ+16,SMK+16,SJ17,BSR18} and the references therein.

Although the results just mentioned cover a wide variety of target tracking scenarios, they do not apply to those where the loss function of interest is \emph{non-convex}. One such scenario is \emph{time-of-arrival (TOA)-based tracking}, in which sensors collect TOA measurements of the target signal and the tracking is achieved by minimizing a non-convex least-squares loss function associated with the measurements collected at each time step~\cite{ZXY+10,XDD13,LTS20}. In this scenario, the tracking problem can be viewed as a sequential version of the well-studied TOA-based source localization problem; see, e.g.,~\cite{CMS04,SY07,BSL08,BTC08,XDD11a,JSZ+13,So19}.
%
%
As far as we know, the TOA-based tracking problem has barely been investigated from the time-varying or online optimization perspective in the literature. Recently, there have been some works that study time-varying optimization problems with general non-convex loss functions. However, the results are not entirely satisfactory when specialized to the TOA-based tracking problem. For instance, the work~\cite{LTS20} proposes an online Newton's method (ONM) and establishes a bound on its \emph{dynamic regret} (i.e., the difference between the cumulative loss incurred by the sequence of solutions generated by the method and that incurred by the sequence of optimal solutions; see~\cite{SDP+20}) by assuming, among other things, that the Hessian of the loss function at each time step satisfies a non-degeneracy condition. It also demonstrates the numerical performance of ONM on the TOA-based tracking problem. Nevertheless, since ONM needs to compute the inverse of the Hessian of the loss function at each time step, it can be computationally expensive. In addition, the work does not shed any light on whether the TOA-based tracking problem satisfies the assumptions underlying the dynamic regret analysis of ONM. As such, the theoretical performance of ONM for TOA-based tracking remains unclear. On the other hand, the work~\cite{HMMR20} develops a dual averaging method and obtains a bound on its dynamic regret under relatively mild assumptions on the loss functions. However, the method is mainly of theoretical interest, as it needs to compute a distribution over the feasible solutions and sample a solution from this distribution at each time step, and neither of these is straightforward to implement for the TOA-based tracking problem.

Motivated by the above discussion, we are interested in developing a low-complexity online method for TOA-based tracking and establishing theoretical guarantee on its performance. One method that naturally suggests itself is online gradient descent (OGD). The method only needs to perform a single gradient descent update at each time step, thus making it well-suited for the target tracking task. However, there has been no performance analysis of OGD for our problem setting so far. Not surprisingly, a major difficulty is that the least-squares loss function associated with the TOA measurements is non-convex. The main contribution of this work is the development of the first non-trivial performance bound for OGD when it is applied to the TOA-based tracking problem. The performance metric we adopt is the \emph{cumulative target tracking error} (CTTE), which is defined as the sum of the distances between the estimated target position and the \emph{true target position} at different time steps. Our bound makes explicit the dependence of the CTTE of OGD on the path length of the target trajectory and the noise power of the TOA measurements. It is important to note that there is a subtle yet fundamental difference in nature between the CTTE metric and most other metrics used in the time-varying or online optimization literature. The former measures the performance relative to the \emph{true values of the parameter} we wish to estimate (viz. the true positions of the target at different time steps), while the latter (such as the dynamic regret or the usual tracking error) measure the performance relative to the \emph{optimal solutions} to the loss minimization problems at different time steps. In the context of the TOA-based tracking problem, it is clear that the CTTE defined above is a more relevant performance metric, as ultimately we are interested in how well the online method tracks the true target positions rather than the optimal solutions to the time-varying loss minimization problem. Nevertheless, the use of the true target positions in the definition of CTTE makes it a more challenging metric to analyze.

To establish the said CTTE bound, we proceed in two steps. First, we revisit the classic least-squares formulation of the (static) TOA-based source localization problem and elucidate its estimation and geometric properties. Specifically, under standard assumptions on the TOA measurement model, we establish a bound on the estimation error of any least-squares estimate of the true target position and use this bound to show that the loss function, albeit non-convex in general, is locally strongly convex at its global minima. Moreover, we give an explicit estimate of the size of the strong convexity region. We remark that similar results have previously been established for a \emph{time-difference-of-arrival} (TDOA)-based least-squares loss function~\cite{LPS17}. However, to the best of our knowledge, our results for the TOA-based least-squares loss function are new and can be of independent interest. In particular, it provides further theoretical justification for the good empirical performance of gradient-based schemes observed in~\cite{BTC08} when solving the TOA-based source localization problem. Second, we extend our local strong convexity result from the static localization setting to the dynamic target tracking setting. Specifically, we show that as long as the aforementioned assumptions on the TOA measurement model are satisfied and the distance between the true positions of the target at consecutive time steps is sufficiently small, the position estimate of the target at the current time step will lie in the strong convexity region of the loss function at the next time step. This allows us to utilize techniques from online strongly convex optimization to establish the advertised CTTE bound for OGD. 

The notation in this paper is mostly standard. We use $\|\cdot\|_1$ and $\|\cdot\|$ to denote the $\ell_1$-norm and Euclidean norm, respectively. Given a vector $\bar{\bm x}\in\R^n$ and a scalar $r>0$, we use $B(\bar{\bm x},r) := \{\bm{x}\in\R^n: \|\bm{x}-\bar{\bm x}\| \le r\}$ to denote the closed Euclidean ball with center $\bar{\bm x}$ and radius $r$. Given a symmetric matrix $\bm{A}$, we use $\lambda_{\min}(\bm{A})$ to denote its smallest eigenvalue and $\bm{A}\succ\bm{0}$ to indicate that it is positive definite.

The rest of the paper is organized as follows. In Section~\ref{sec:formulation}, we present a time-varying optimization formulation of the TOA-based tracking problem and describe how it can be tackled by OGD. In Section~\ref{sec:str_cvx}, we study the estimation error and local strong convexity property of the static TOA-based source localization problem. Using these results, we establish our bound on the CTTE of OGD for the TOA-based tracking problem in Section~\ref{sec:toa-reg}. In Section~\ref{sec:sim}, we present numerical results to demonstrate the efficacy of OGD for TOA-based tracking and illustrate our theoretical findings. We then end with some closing remarks in Section~\ref{sec:concl}.

\section{Problem Formulation and Preliminaries}\label{sec:formulation}
We begin by describing the setup for TOA-based tracking. Let $\bm{x}_t^{\star}\in\mathbb{R}^n$ be the unknown true position of the moving target at time $t$, where $t=1,\ldots,T$ and $T$ is the time horizon of interest. Furthermore, let $\bm{a}_i\in\mathbb{R}^n$ ($i=1,\ldots,m$) be the known position of the $i$th sensor and suppose that the vectors $\{\bm{a}_i-\bm{a}_1\}_{i=2}^m$ span $\mathbb{R}^n$ (in particular, we have $m \ge n+1$). We consider the following model for TOA-based range measurements:
\begin{equation} \label{eq:toa-model}
r_i^t = \|\bm{x}_t^{\star} - \bm{a}_i\| + w_i^t, \quad i=1,\ldots,m; \, t=1,\ldots,T.
\end{equation}
Here, $w_i^t$ is the measurement noise and $r_i^t$ is the noisy TOA-based range measurement between the target and the $i$th sensor at time $t$. We assume that $w_i^t$ is a random variable with mean zero, variance bounded above by $\sigma_t^2$ and is independent of the noise at other sensors and at other time steps. We also assume that $|w_i^t| \ll \|\bm{x}_t^{\star} - \bm{a}_i\|$ for $i=1,\ldots,m$ and $t=1,\ldots,T$. It is worth noting that similar assumptions have appeared in the localization literature; see, e.g.,~\cite{WSL16}. 

To estimate the target position at time $t$, a natural approach is to consider the following non-convex least-squares formulation:
\begin{equation} \label{eq:loss}
\min_{\bm{x} \in \R^n} \ f_t(\bm{x}) := \sum_{i=1}^m(\|\bm{x} - \bm{a}_i\| - r_i^t)^2.
\end{equation}
Such a formulation is motivated by the fact that when the measurement noise vector $\bm{w}^t=(w_1^t,\ldots,w_m^t)$ is Gaussian, every optimal solution to Problem~\eqref{eq:loss} is a maximum-likelihood estimate of the true target position $\bm{x}_t^\star$; see, e.g.,~\cite{CMS04}. Henceforth, we shall use $\hat{\bm x}_t$ to denote an optimal solution to~\eqref{eq:loss} (i.e., $\hat{\bm x}_t \in \arg\min_{\bm{x}\in\R^n} f_t(\bm{x})$) and refer to it as a \emph{least-squares estimate} of the true target position $\bm{x}_t^\star$. In this paper, we propose to apply OGD to tackle the time-varying optimization formulation~\eqref{eq:loss}, as it may not be computationally feasible to find an (approximately) optimal solution to~\eqref{eq:loss} at every time step. Specifically, given an estimate $\bm{x}_{t-1}$ of the target position at time $t-1$ and the noisy range measurements $\{r_i^t\}_{i=1}^m$ at time $t$, we generate an estimate $\bm{x}_t$ of the target position at time $t$ via the one-step gradient descent update
\begin{align}\label{eq:toa-ogd}
\bm{x}_{t} = \bm{x}_{t-1} - \eta_t\nabla f_t(\bm{x}_{t-1}), \quad t=1,\ldots,T,
\end{align}
where $\eta_t>0$ is the step size. We remark that the update~\eqref{eq:toa-ogd} should be interpreted in a formal sense at this point, as the function $f_t$ is non-differentiable at $\bm{x} \in \{\bm{a}_1,\ldots,\bm{a}_m\}$. We shall justify the validity of~\eqref{eq:toa-ogd} in the following sections.

Naturally, we are interested in evaluating the performance of the sequence of position estimates $\{\bm{x}_t\}_{t=1}^T$. For that purpose, we employ the notion of CTTE, which is defined as
\[ 
{\rm CTTE}\left( \{\bm{x}_t\}_{t=1}^T \right) := \sum_{t=1}^T \| \bm{x}_t - \bm{x}_t^\star \|.
\] 
Note that the definition of CTTE involves the sequence of \emph{true target positions} $\{\bm{x}_t^\star\}_{t=1}^T$, not the sequence of \emph{optimal solutions} $\{\hat{\bm{x}}_t\}_{t=1}^T$ to Problem~\eqref{eq:loss}, as it is the former that we are interested in tracking. Indeed, a small CTTE implies that the estimate $\bm{x}_t$ is close to the true target position $\bm{x}_t^\star$ at every time step $t$. Our goal is to bound the CTTE in terms of the variations in the target trajectory $\{ \|\bm{x}_{t+1}^\star - \bm{x}_t^\star\| \}_{t=1}^{T-1}$ and the noise power $\{\sigma_t^2\}_{t=1}^T$ and to derive conditions that can guarantee a sublinear CTTE bound (i.e., $\tfrac{1}{T}{\rm CTTE}\left( \{\bm{x}_t\}_{t=1}^T \right) \rightarrow 0$) on the tracking performance of OGD. We remark that a sublinear CTTE bound is a desirable property for a tracking algorithm to have, as it implies that the target tracking error of the algorithm---i.e., the distance between the target position estimate produced by the algorithm and the true target position---vanishes asymptotically. In the next section, we will develop two results that are key to achieving this goal. Specifically, under the assumption that the power of the measurement noise $\sigma_t^2$ is sufficiently small, we will first establish a bound on the estimation error $\|\hat{\bm x}_t - \bm{x}_t^\star\|$ and then use this bound to show that the loss function $f_t$ is locally strongly convex at the least-squares estimate $\hat{\bm x}_t$.\footnote{A function $g:\R^n\limto\R$ is said to be \emph{locally strongly convex at $\bar{\bm x}$} if there exists an $r>0$ such that $g$ is strongly convex on the ball $B(\bar{\bm x},r)$~\cite{Vial82}.}

\section{Local Strong Convexity of TOA-Based Source Localization} \label{sec:str_cvx}
Consider a fixed time $t$. Then, Problem~\eqref{eq:loss} reduces to the classic TOA-based source localization problem (see, e.g.,~\cite{So19}), in which the target is considered static. For notational simplicity, we drop the index $t$ and write Problem~\eqref{eq:loss} as
\begin{align}\label{eq:toa-ml}
\min_{\bm{x}\in\R^n} \  f(\bm{x}) := \sum_{i=1}^m(\|\bm{x} - \bm{a}_i\| - r_i)^2
\end{align}
with $r_i=\|\bm{x}^{\star} - \bm{a}_i\| + w_i$. As before, we assume that $w_i$ is a random variable with mean zero, variance bounded above by $\sigma^2$ and satisfies $|w_i| \ll \| \bm{x}^{\star} - \bm{a}_i \|$. Let $\hat{\bm{x}} \in \arg\min_{\bm{x}\in\R^n} f(\bm{x})$ denote a least-squares estimate of the true target position $\bm{x}^\star$. The following proposition, which plays a crucial role in our subsequent development, shows that $\hat{\bm x}$ and $\bm{x}^\star$ are close when the power of the measurement noise vector $\bm{w}=(w_1,\ldots,w_m)$ is small.
\begin{proposition}[Estimation Error of Least-Squares Estimator] \label{thm:esterror}
Suppose that $\|\bm{w}\| \le c_0 \sqrt{m}\sigma$ for some constant $c_0>0$. Then, there exist constants $K_1,~K_2 > 0$, which are determined by $\bm{a}_1,\ldots,\bm{a}_m$ and $\bm{x}^\star$, such that
\[ 
\|\hat{\bm{x}}-\bm{x}^\star\| \leq K_1\sqrt{m}\sigma + K_2m\sigma^2.
\] 
\end{proposition}
\noindent The proof of Proposition~\ref{thm:esterror} can be found in Appendix~\ref{app:esterror}.

The assumption on $\|\bm{w}\|$ in Proposition~\ref{thm:esterror} is rather mild, as it can be satisfied with high probability when, e.g.,  $w_1,\ldots,w_m$ are sub-Gaussian random variables~\cite[Chapter 3]{Ver18}. Now, using Proposition~\ref{thm:esterror}, we can prove the following theorem, which establishes the local strong convexity of $f$ at $\hat{\bm x}$ and provides an explicit estimate on the size of the strong convexity region around $\hat{\bm x}$. This constitutes our first main result in this paper.
\begin{theorem}[Local Strong Convexity of TOA-Based Source Localization] \label{thm:str_cvx}
Consider the setting of Proposition~\ref{thm:esterror}. Suppose that for some given $\delta>0$, the noise power $\sigma^2$ satisfies
\begin{equation} \label{eq:dist-asp}
\|\bm{x}^{\star} - \bm{a}_i\|>K_1\sqrt{m}\sigma + K_2m\sigma^2+\delta, \quad i=1,\ldots,m
\end{equation}
and
\begin{align}
\kappa &:= \frac{\delta}{10m} \cdot \Lambda - (K_1\sqrt{m}\sigma + K_2m\sigma^2) - \frac{4c_0\sigma}{5} > 0, \label{eq:eps}
\end{align}
where
\[ \Lambda := \lambda_{\min}\left( \sum_{i=1}^m \left( \frac{\bm{x}^\star-\bm{a}_i}{\|\bm{x}^\star-\bm{a}_i\|} \right)\left( \frac{\bm{x}^\star-\bm{a}_i}{\|\bm{x}^\star-\bm{a}_i\|}\right)^T \right). \]
Then, we have $\nabla^2f(\hat{\bm x}+\bm{\epsilon}) \succ \bm{0}$ for all $\bm{\epsilon} \in \R^n$ satisfying $\| \bm{\epsilon} \| \le \kappa$; i.e., $f$ is strongly convex over $B(\hat{\bm x},\kappa)$.
\end{theorem}
\noindent The proof of Theorem~\ref{thm:str_cvx} can be found in Appendix~\ref{app:str_cvx}. Here, let us elaborate on the assumptions of the theorem.
\begin{enumerate}
\item Condition~\eqref{eq:dist-asp} stipulates that the target should be sufficiently far from the sensors, which is not very restrictive in practice. Moreover, when combined with Proposition~\ref{thm:esterror}, the condition implies that $\| \hat{\bm x} - \bm{a}_i\| > \delta$ for $i=1,\ldots,m$, which shows that the loss function $f$ is smooth around the least-squares estimate $\hat{\bm x}$. This allows us to use the Hessian $\nabla^2f$ to characterize the local strong convexity of $f$ at $\hat{\bm x}$.

\item Since the vectors $\{\bm{a}_i-\bm{a}_1\}_{i=2}^m$ span $\R^n$ by assumption, it can be shown that the vectors $\{\bm{x}^\star - \bm{a}_i\}_{i=1}^m$ also span $\R^n$. This implies that $\Lambda > 0$. Thus, condition~\eqref{eq:eps} can be satisfied when $\sigma>0$ is sufficiently small (incidentally, condition~\eqref{eq:dist-asp} also becomes easier to satisfy as $\sigma$ becomes smaller). An important insight drawn from~\eqref{eq:eps} is that the landscape of the loss function $f$ around the least-squares estimate $\hat{\bm x}$ depends on the noise power level and the geometric configuration of the target and sensors.
\end{enumerate}

We remark that although the TOA-based source localization problem has been extensively studied in the literature, Theorem~\ref{thm:str_cvx} is, to the best of our knowledge, the first result that elicits the local strong convexity property of the non-convex least-squares formulation~\eqref{eq:toa-ml}. Now, since the strong convexity region $B(\hat{\bm x},\kappa)$ around $\hat{\bm x}$ is compact and $\nabla^2f$ is continuous over $B(\hat{\bm x},\kappa)$, we see that $\nabla f$ is Lipschitz continuous over $B(\hat{\bm x},\kappa)$. Thus, Theorem~\ref{thm:str_cvx} implies that when applying the gradient descent method to tackle Problem~\eqref{eq:toa-ml}, the resulting sequence of iterates will converge to the optimal solution $\hat{\bm x}$ at a linear rate, provided that the initial point lies in the strong convexity region around $\hat{\bm x}$. This can be deduced using the following well-known result.
\begin{fact}[Linear Convergence of Gradient Descent for Strongly Convex Minimization; cf.~{\cite[Theorem 2.1.15]{N04}}]\label{thm:conv_GD} 
Let $g:\mathbb{R}^n\rightarrow\mathbb{R}$ be a function that is smooth, $\mu$-strongly convex, and $L$-gradient Lipschitz continuous on an open convex set $\mathcal{X}\subseteq\mathbb{R}^n$. Suppose that $g$ has a global minimizer $\hat{\bm x}$ over $\mathcal{X}$. Then, the sequence $\{\bm{x}_k\}_{k\ge0}$ generated by the gradient descent method 
\[ \bm{x}_{k+1} = \bm{x}_k - \eta\nabla g(\bm{x}_k) \]
with initial point $\bm{x}_0 \in \mathcal{X}$ and step size $\eta \in (0,2/(\mu + L)]$ satisfies
\[
\|\bm{x}_{k+1} - \hat{\bm{x}}\|^2 \leq\left(1 - \frac{2\eta\mu L}{\mu+L}\right)\|\bm{x}_k - \hat{\bm{x}}\|^2.
\]
\end{fact}
\noindent In particular, Theorem~\ref{thm:str_cvx} provides a means to justify the good empirical performance of gradient-based schemes observed in~\cite{BTC08} when solving the TOA-based source localization problem.

\section{CTTE of OGD for TOA-Based Tracking}\label{sec:toa-reg}
Let us now address the main goal of this paper---namely, to establish a bound on the CTTE of OGD for TOA-based tracking. The results in Section~\ref{sec:str_cvx} suggest that if the iterate generated by OGD at time $t$ lies in the strong convexity region of the loss function at time $t+1$ for $t=0,1,\ldots,T-1$, then the tracking problem is essentially reduced to that of minimizing a time-varying strongly convex function. This opens up the possibility of using techniques from online strongly convex optimization to bound the CTTE of OGD for TOA-based tracking.

To realize the above idea, we need to first introduce some additional preliminaries and collect some consequences of the results in Section~\ref{sec:str_cvx}. Observe that the constants $K_1, K_2, \Lambda$ in Theorem~\ref{thm:str_cvx} involve the target position $\bm{x}^\star$. Since the target is moving in the tracking setting, it will simplify our subsequent analysis if we can find uniform bounds on these constants. Towards that end, we further assume that the target stays within a fixed compact region $\mathcal{T} \subseteq \R^n$ throughout the tracking task. Such an assumption is rather mild in practice. Moreover, since $K_1,K_2,\Lambda$ depend continuously on $\bm{x}^\star$, it implies the existence of finite upper bounds on $K_1, K_2$ and a positive lower bound on $\Lambda$ that hold for all $t\ge0$. As a slight abuse of notation, we shall use $K_1,K_2,\Lambda$ to denote these uniform bounds in the sequel.

Following the setting of Theorem~\ref{thm:str_cvx}, let $\hat{\bm x}_t \in \arg\min_{\bm{x}\in\R^n} f_t(\bm{x})$ denote a least-squares estimate of the true target position $\bm{x}_t^\star$ at time $t$ and $c_0>0$ be a constant such that $\|\bm{w}^t\| \le c_0\sqrt{m}\sigma_t$ for $t=1,\ldots,T$. Furthermore, suppose that for some given $\delta>0$, the maximum noise power $\sigma^2 := \max_{t\in\{1,\ldots,T\}} \sigma_t^2$ satisfies
\begin{align} 
\|\bm{x}_t^{\star} - \bm{a}_i\| &> K_1\sqrt{m}\sigma + K_2m\sigma^2+\delta, \nonumber \\
&\qquad\qquad i=1,\ldots,m; \, t=1,\ldots,T \label{eq:dist-asp-dyn}
\end{align}
and
\begin{align}
\kappa &:= \frac{\delta}{10m} \cdot \Lambda - (K_1\sqrt{m}\sigma + K_2m\sigma^2) - \frac{4c_0\sigma}{5} > 0 \label{eq:eps-dyn}
\end{align}
(recall from the discussion in the preceding paragraph that $K_1,K_2,\Lambda$ are now uniform in $t$ and hence $\kappa$ is also uniform in $t$). Then, using Theorem~\ref{thm:str_cvx}, the expressions for $\nabla f_t, \nabla^2 f_t$, and the assumption that the target stays within the compact region $\mathcal{T}$, we deduce the existence of constants $\mu,L>0$ such that for $t=1,\ldots,T$,
\begin{enumerate}
\item $f_t$ is $\mu$-strongly convex over $B(\hat{\bm{x}}_t,\kappa)$---i.e., for any $\bm{x},\bm{y}\in B(\hat{\bm{x}}_t,\kappa)$,
\begin{align}\label{eq:toa-strcvx}
f_t(\bm{x}) \geq f_t(\bm{y}) +\nabla f_t(\bm{y})^T(\bm{x}-\bm{y})+\frac{\mu}{2}\|\bm{x}-\bm{y}\|^2;
\end{align}

\item $\nabla f_t$ is $L$-Lipschitz continuous over $B(\hat{\bm{x}}_t,\kappa)$---i.e., for any $\bm{x},\bm{y}\in B(\hat{\bm{x}}_t,\kappa)$,
\begin{align}\label{eq:toa-gradlip}
\|\nabla f_t(\bm{y}) - \nabla f_t(\bm{x})\|\leq L\|\bm{x}-\bm{y}\|;
\end{align}
\end{enumerate}

Now, let $\{\bm{x}_t\}_{t=1}^T$ be the sequence of iterates generated by the OGD update~\eqref{eq:toa-ogd} with initial point $\bm{x}_0$ and step size $\eta_t \equiv \eta \in (0,2/(\mu+L)]$ for $t=1,\ldots,T$. In addition, let $v_t := \|\bm{x}_{t+1}^\star - \bm{x}_t^\star\|$ ($t=1,\ldots,T-1$) denote the variation in the true target position between time $t$ and $t+1$ and $v := \max_{t\in\{1,\ldots,T-1\}} v_t$ denote the maximum variation in the true target position between successive time steps. The following proposition shows that under suitable conditions, OGD maintains the invariant that the iterate generated at the current time step lies in the strong convexity region of the loss function at the next time step.
\begin{proposition}[Invariant of OGD] \label{prop:ogd-inv}
Suppose that in addition to~\eqref{eq:dist-asp-dyn} and~\eqref{eq:eps-dyn}, the maximum noise power $\sigma^2$ and maximum variation $v$ satisfy
\begin{equation} \label{eq:eps-rad}
\kappa \ge \frac{2(K_1\sqrt{m}\sigma + K_2m\sigma^2) + v}{1-\rho},
\end{equation}
where $\rho := \left( 1-\tfrac{2 \eta \mu L}{\mu+L} \right)^{1/2} \in (0,1)$ with $\mu,L$ given by~\eqref{eq:toa-strcvx},~\eqref{eq:toa-gradlip}, respectively, and $\kappa > 0$ is the radius of the strong convexity region of the loss function $f_t$ around the least-squares estimate $\hat{\bm x}_t$ for $t=1,\ldots,T$. Furthermore, suppose that the initial point $\bm{x}_0$ satisfies $\|\bm{x}_0 - \bm{x}_1^\star\| \le K_1\sqrt{m}\sigma + K_2m\sigma^2$. Then, for $t=0,1,\ldots,T-1$, the iterate $\bm{x}_t$ lies in the strong convexity region $B(\hat{\bm x}_{t+1},\kappa)$ of the loss function $f_{t+1}$.
\end{proposition}
\begin{proof}
We proceed by induction on $t$. For $t=0$, we have
\begin{align}
\| \bm{x}_0 - \hat{\bm x}_1 \| &\le \| \bm{x}_0 - \bm{x}_1^\star \| + \| \bm{x}_1^\star - \hat{\bm x}_1 \| \nonumber \\
&\le 2(K_1\sqrt{m}\sigma + K_2m\sigma^2) \label{eq:init-bd} \\
&\le \kappa, \nonumber
\end{align}
where the second inequality follows from our assumption on $\bm{x}_0$ and Proposition~\ref{thm:esterror} and the last follows from our choice of $\kappa$ in~\eqref{eq:eps-rad}. This establishes the base case. Now, for $t\ge1$, we have
\begin{align*}
&~ \| \bm{x}_t - \hat{\bm x}_{t+1} \| \le \| \bm{x}_t - \hat{\bm x}_t \| + \| \hat{\bm x}_t - \hat{\bm x}_{t+1} \| \\
\le&~ \rho \| \bm{x}_{t-1} - \hat{\bm x}_t \| + \| \hat{\bm x}_t - \bm{x}_t^\star \| + \| \bm{x}_{t+1}^\star - \hat{\bm x}_{t+1} \| \\
&\quad~ + \| \bm{x}_t^\star - \bm{x}_{t+1}^\star \| \\
\le&~ \rho\kappa + 2(K_1\sqrt{m}\sigma + K_2m\sigma^2) + v_t \\
\le&~ \kappa,
\end{align*} 
where the second inequality follows from the OGD update~\eqref{eq:toa-ogd}, the inductive hypothesis (i.e., $\bm{x}_{t-1}$ lies in the strong convexity region of $f_t$), and Fact~\ref{thm:conv_GD}; the third follows from the inductive hypothesis and Proposition~\ref{thm:esterror}; the last follows from our choice of $\kappa$ in~\eqref{eq:eps-rad}. This completes the inductive step and also the proof of Proposition~\ref{prop:ogd-inv}.
\end{proof}

We remark that since the loss functions $\{f_t\}_{t=1}^T$ are non-convex, some conditions on the maximum noise power, maximum variation, and quality of the initial point are to be expected in the CTTE analysis of OGD for tackling the TOA-based tracking problem~\eqref{eq:loss}. In fact, the performance analysis of ONM for general time-varying non-convex optimization in~\cite{LTS20}, though focusing on the dynamic regret metric, makes use of similar conditions on the maximum variation and quality of the initial point as those in Proposition~\ref{prop:ogd-inv}.

Armed with Proposition~\ref{prop:ogd-inv}, we can prove the following theorem, which establishes a CTTE bound for OGD when it is applied to the TOA-based tracking problem. This constitutes our second main result in this paper.
\begin{theorem}[CTTE of OGD for TOA-Based Tracking] \label{thm:ogd-ctte}
Under the setting of Proposition~\ref{prop:ogd-inv}, the sequence of iterates $\{\bm{x}_t\}_{t=1}^T$ satisfies 
\[ {\rm CTTE}\left( \{\bm{x}_t\}_{t=1}^T \right) = \mathcal{O}(1 + V(T) + N_1(T) + N_2(T)), \]
where $V(T) := \sum_{t=1}^{T-1} \| \bm{x}_{t+1}^\star - \bm{x}_t^\star \| = \sum_{t=1}^{T-1} v_t$ denotes the path length of the target trajectory, $N_1(T):=\sum_{t=1}^T \sigma_t$ denotes the cumulative noise standard deviation, and $N_2(T):=\sum_{t=1}^T \sigma_t^2$ denotes the cumulative noise variance. 
\end{theorem}
\begin{proof}
Using the definition of CTTE and the triangle inequality, we have
\begin{align}
{\rm CTTE}\left( \{\bm{x}_t\}_{t=1}^T \right) &= \sum_{t=1}^T \| \bm{x}_t - \bm{x}_t^\star \| \nonumber \\
&\le \sum_{t=1}^T \| \bm{x}_t - \hat{\bm x}_t \| + \sum_{t=1}^T \| \hat{\bm x}_t - \bm{x}_t^\star \|. \label{eq:pre-ctte}
\end{align} 
Let us now bound the two terms in~\eqref{eq:pre-ctte} separately.

For the first term, we begin by adapting the argument used in the proof of~\cite[Theorem 1]{MSJ+16} to our time-varying optimization setting and bound
\begin{align*}
&~ \sum_{t=1}^T \| \bm{x}_t - \hat{\bm x}_t \| \le \rho \sum_{t=1}^T \| \bm{x}_{t-1} - \hat{\bm x}_t \| \\
\le&~ \rho \| \bm{x}_0 - \hat{\bm x}_1 \| + \rho \sum_{t=2}^T \| \bm{x}_{t-1} - \hat{\bm x}_{t-1} \| + \rho \sum_{t=2}^T \| \hat{\bm x}_{t-1} - \hat{\bm x}_t \| \\
=&~ \rho \left( \| \bm{x}_0 - \hat{\bm x}_1 \|  - \| \bm{x}_T - \hat{\bm x}_T \| \right) + \rho \sum_{t=1}^T \| \bm{x}_t - \hat{\bm x}_t \| \\
&~\, + \rho \sum_{t=1}^{T-1} \| \hat{\bm x}_t - \hat{\bm x}_{t+1} \|,
\end{align*}
where the first inequality follows from the OGD update~\eqref{eq:toa-ogd}, Proposition~\ref{prop:ogd-inv}, and Fact~\ref{thm:conv_GD}. It follows that
\begin{align}
\sum_{t=1}^T \| \bm{x}_t - \hat{\bm x}_t \| \le \frac{\rho}{1-\rho} \left( \| \bm{x}_0 - \hat{\bm x}_1 \|  + \sum_{t=1}^{T-1} \| \hat{\bm x}_t - \hat{\bm x}_{t+1} \| \right). \label{eq:cum-err}
\end{align}
Now, using~\eqref{eq:init-bd}, we get
\[ \| \bm{x}_0 - \hat{\bm x}_1 \| \le 2(K_1\sqrt{m}\sigma+K_2m\sigma^2). \]
Furthermore, we have
\begin{align*}
&~ \sum_{t=1}^{T-1} \| \hat{\bm x}_t - \hat{\bm x}_{t+1} \|  \\
\le&~ \sum_{t=1}^{T-1} \left( \| \hat{\bm x}_t - \bm{x}_t^\star \| + \| \bm{x}_t^\star - \bm{x}_{t+1}^\star \| + \| \bm{x}_{t+1}^\star - \hat{\bm x}_{t+1} \| \right) \\
\le&~ \sum_{t=1}^{T-1} \left( K_1\sqrt{m}(\sigma_t+\sigma_{t+1}) + K_2m(\sigma_t^2+\sigma_{t+1}^2) + v_t \right) \\
=&~ \sum_{t=1}^{T-1} v_t + 2K_1\sqrt{m} \sum_{t=1}^{T} \sigma_t + 2K_2m\sum_{t=1}^{T} \sigma_t^2,
\end{align*}
where the second inequality follows from Proposition~\ref{thm:esterror}. Substituting the above into~\eqref{eq:cum-err} yields 
\[ \sum_{t=1}^T \| \bm{x}_t - \hat{\bm x}_t \| = \mathcal{O}(1 + V(T) + N_1(T) + N_2(T)). \]

For the second term, we simply invoke Proposition~\ref{thm:esterror} to get
\begin{align*}
\sum_{t=1}^T \| \hat{\bm x}_t - \bm{x}_t^\star \| &\le K_1\sqrt{m} \sum_{t=1}^T \sigma_t + K_2 m \sum_{t=1}^T \sigma_t^2 \\
&= \mathcal{O}(N_1(T) + N_2(T)).
\end{align*}

The desired result now follows by substituting the above into~\eqref{eq:pre-ctte}.
\end{proof}

Theorem~\ref{thm:ogd-ctte} reveals that OGD can achieve sublinear CTTE when both the path length $V(T)$ and the cumulative noise power $N_2(T)$ grow sublinearly (note that the latter, together with the fact that $N_1(T) \le \sqrt{T \cdot N_2(T)}$, implies the sublinear growth of the cumulative noise standard deviation $N_1(T)$). Roughly speaking, this means that if the target is not moving too fast and the noise power decays at a sufficiently fast rate over time, then the target tracking error of OGD will vanish asymptotically. It is important to note that our CTTE bound is expressed in terms of the path length of the \emph{target trajectory} (i.e., $V(T) = \sum_{t=1}^{T-1} \| \bm{x}_{t+1}^\star - \bm{x}_t^\star \|$), not the path length of the \emph{optimal solution trajectory} of the time-varying loss function (i.e., $V'(T) := \sum_{t=1}^{T-1} \| \hat{\bm x}_{t+1} - \hat{\bm x}_t \|$). Although the latter is commonly used in existing performance analyses of online methods (see, e.g.,~\cite{MSJ+16,BSR18,LTS20}), the former captures the actual variations in the target trajectory and is thus more relevant to the tracking problem considered in this paper. It is also worth noting that our CTTE bound shows explicitly how the TOA measurement noise affects the tracking performance of OGD through the terms $N_1(T)$ and $N_2(T)$.

\section{Numerical Simulations} \label{sec:sim}
In this section, we present numerical results to demonstrate the efficacy of OGD for the TOA-based tracking problem and illustrate our theoretical findings. Specifically, we apply both OGD and ONM---the latter has previously been used in~\cite{LTS20} to tackle the TOA-based tracking problem---to various test instances and compare their tracking performance. In all the considered instances, there are $m=3$ sensors, which are located at
$\bm{a}_1 = 
\begin{bmatrix}
0.5 & 0.5
\end{bmatrix}
^T$, 
$\bm{a}_2 = 
\begin{bmatrix}
0 & 0.5
\end{bmatrix}
^T$,
and $\bm{a}_3 = 
\begin{bmatrix}
0.5 & 0
\end{bmatrix}
^T$. Given the time horizon of interest $T$ and the target trajectory $\{\bm{x}_t^\star\}_{t=1}^T$, the measurement noise $w_i^t$ in~\eqref{eq:toa-model} is generated according to the Gaussian distribution with mean zero and variance $\sigma_t^2$ for $i=1,\ldots,m$; $t=1,\ldots,T$, and the TOA-based range measurements $\{ r_i^t : i=1,\ldots,m; \, t=1,\ldots,T \}$ are then obtained using~\eqref{eq:toa-model}. We consider two initialization strategies for OGD and ONM. One is \emph{exact initialization}, which assumes that the true initial target position $\bm{x}_1^\star$ is known and takes $\bm{x}_0=\bm{x}_1^\star$ as the initial point. The other is \emph{ordinary least-squares (OLS) initialization}, which takes 
\begin{equation} \label{eq:OLS-init}
\bm{x}_0 = (\bm{A}^T\bm{A})^{-1}\bm{A}^T\bm{b}_1
\end{equation}
with
\begin{align}
\bm{A} &:=
\begin{bmatrix}
(\bm{a}_2-\bm{a}_1)^T\\
\vdots\\
(\bm{a}_m-\bm{a}_{m-1})^T\\
\end{bmatrix}, \label{eq:LS-A} \\
\bm{b}_1 &:= 
\frac{1}{2}\begin{bmatrix}
\|\bm{a}_2\|^2 - \|\bm{a}_1\|^2 + (r_1^1)^2 - (r_2^1)^2 \\
\vdots\\
\|\bm{a}_m\|^2 - \|\bm{a}_{m-1}\|^2 + (r_{m-1}^1)^2 - (r_m^1)^2
\end{bmatrix} \label{eq:LS-b}
\end{align}
as the initial point; see~\cite{STK05}. The OLS estimate in~\eqref{eq:OLS-init} can be obtained as follows: Observe that any $\bm{x}$ satisfying
\[ \| \bm{x} - \bm{a}_i \|^2 \approx (r_i^1)^2, \quad i=1,\ldots,m \]
can serve as an estimate of the true initial target position $\bm{x}_1^\star$. Upon subtracting the $i$th equation from the $(i+1)$st, where $i=1,\ldots,m-1$, we get
\[ 2(\bm{a}_{i+1}-\bm{a}_i)^T\bm{x} \approx \|\bm{a}_{i+1}\|^2 - \|\bm{a}_i\|^2 + (r_i^1)^2 - (r_{i+1}^1)^2.\]
In particular, we can obtain an estimate of $\bm{x}_1^\star$ by solving
\begin{align} \label{eq:LS}
\min_{\bm{x} \in \mathbb{R}^n} \|\bm{A}\bm{x} - \bm{b}\|^2,
\end{align}
where $\bm{A}$ and $\bm{b}$ are given by~\eqref{eq:LS-A} and~\eqref{eq:LS-b}, respectively. Since the vectors $\{\bm{a}_{i}-\bm{a}_1\}_{i=2}^{m}$ span $\mathbb{R}^n$ by assumption, the solution to~\eqref{eq:LS} is readily given by~\eqref{eq:OLS-init}. It is worth noting that the OLS estimate in~\eqref{eq:OLS-init} can be computed simply by using the sensor positions $\{\bm{a}_i\}_{i=1}^m$ and noisy range measurements $\{r_i^1\}_{i=1}^m$. Thus, it is an attractive choice for initializing OGD and ONM. We use the step size $\eta_t = 0.1$ for $t=1,\ldots,T$ in OGD. Then, OGD generates the position estimates of the target via~\eqref{eq:toa-ogd}, while ONM generates those via 
\[
\bm{x}_t = \bm{x}_{t-1} - \left(\nabla^2f_t(\bm{x}_{t-1})\right)^{-1}\nabla f_t(\bm{x}_{t-1}), \quad t=1,\ldots,T.
\]
All computations were carried out in MATLAB on an Intel(R) Core(TM) i5-8600 CPU 3.10GHz CPU machine. The CTTE shown in the figures are averaged over 1000 Monte Carlo runs.

\subsection{Small Noise Level and Path Variation}
To begin, we construct the following set of test instances (cf.~\cite[Section IV]{LTS20}): The time horizon of interest $T$ is set to $500$. The target’s initial position is set to $\bm{x}_1^\star = 
\begin{bmatrix}
2 & 1
\end{bmatrix}^T
$ and its positions at subsequent time steps are given by
\begin{equation}\label{eq:source-update}
    \bm{x}_{t+1}^\star = \bm{x}_t^\star + \frac{0.005}{\sqrt{2(t+1)}} \bm{u}_t, \quad t=1,\ldots,T-1,
\end{equation}
where $\bm{u}_1,\ldots,\bm{u}_{T-1} \in \R^2$ are independently and uniformly distributed on the unit circle centered at the origin. We consider three scenarios, which correspond to three different noise levels: (i) $\sigma_t = 0.0001$ for $t=1,\ldots,T$; (ii) $\sigma_t = 0.01$ for $t=1,\ldots,T$; (iii) $\sigma_t = \tfrac{0.01}{\sqrt{t}}$ for $t=1,\ldots,T$. Figures~\ref{fig:error_sigma1e-4}--\ref{fig:error_sigma1e-2oversqrt} show the CTTE of OGD and ONM with exact and OLS initialization at these three noise levels. Figures~\ref{fig:traj_sigma1e-4}--\ref{fig:traj_sigma1e-2oversqrt} show the tracking trajectories generated by OGD and ONM for particular instances at those noise levels with OLS initialization. We also include the trajectories of the least-squares estimates $\{\hat{\bm x}_t\}_{t=1}^T$ in the figures for reference. These trajectories are generated using gradient descent (GD) at each time step. Specifically, at time $t$, we use the true target position $\bm{x}_t^\star$ as the initial point and perform the GD updates using the constant step size $1/m$ until either the norm of the gradient is smaller than $10^{-8}$ or the number of iterations reaches 5000. We then declare the last iterate to be $\hat{\bm x}_t$.

\begin{figure*}[!t]
\centering
\begin{subfigure}{.32\textwidth}
\centering
\includegraphics[width=1.08\textwidth]{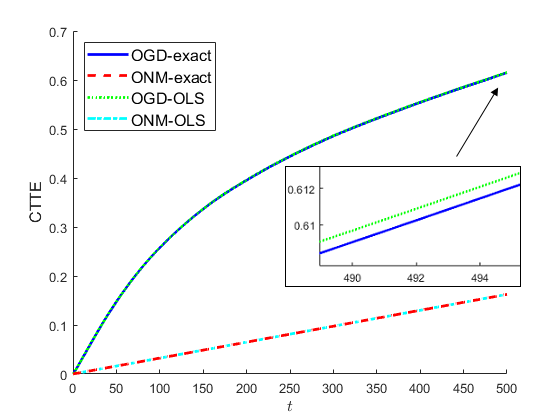}
\caption{$\sigma_t=0.0001$}
\label{fig:error_sigma1e-4}
\end{subfigure}
\begin{subfigure}{.32\textwidth}
\centering
\includegraphics[width=1.08\textwidth]{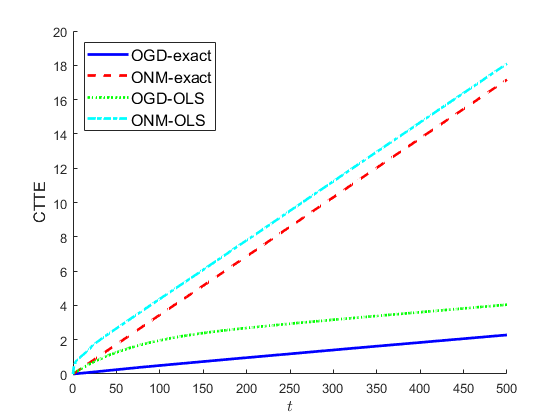}
\caption{$\sigma_t=0.01$}
\label{fig:error_sigma1e-2}
\end{subfigure}
\begin{subfigure}{.32\textwidth}
\centering
\includegraphics[width=1.08\textwidth]{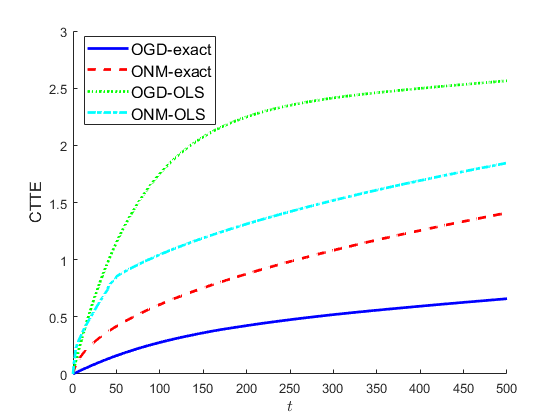}
\caption{$\sigma_t=0.01/\sqrt{t}$}
\label{fig:error_sigma1e-2oversqrt}
\end{subfigure}
\vfill
\begin{subfigure}{.32\textwidth}
\centering
\includegraphics[width=1.08\textwidth]{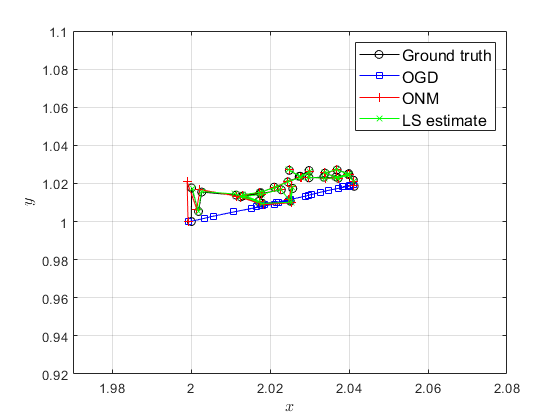}
\caption{$\sigma_t=0.0001$}
\label{fig:traj_sigma1e-4}
\end{subfigure}
\begin{subfigure}{.32\textwidth}
\centering
\includegraphics[width=1.08\textwidth]{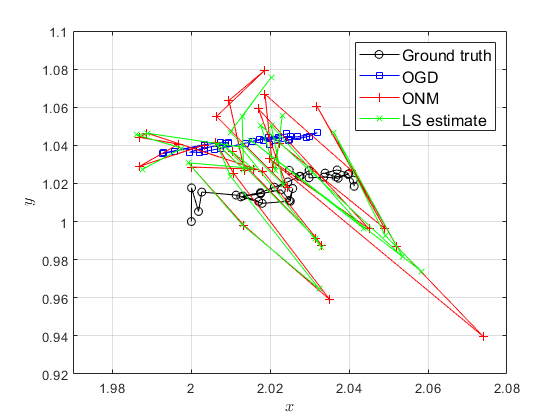}
\caption{$\sigma_t=0.01$}
\label{fig:traj_sigma1e-2}
\end{subfigure}
\begin{subfigure}{.32\textwidth}
\centering
\includegraphics[width=1.08\textwidth]{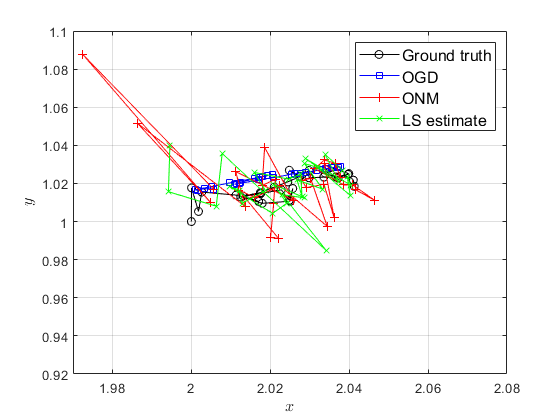}
\caption{$\sigma_t=0.01/\sqrt{t}$}
\label{fig:traj_sigma1e-2oversqrt}
\end{subfigure}
\caption{CTTE (top row) and tracking trajectories (bottom row) of OGD and ONM at different noise levels.}
\label{fig:OGDvsONM}
\end{figure*}

In the first scenario, the noise level is small compared to the path variation (i.e., $\sigma_{t+1}=0.0001$ vs. $v_t = \tfrac{0.005}{\sqrt{2(t+1)}}$ for $t=1,\ldots,T-1$ with $T=500$). We see from Figure~\ref{fig:error_sigma1e-4} that ONM has a smaller CTTE than OGD. This can be explained as follows: First, Proposition~\ref{thm:esterror} implies that the least-squares estimate $\hat{\bm x}_t$ is close to the true target position $\bm{x}_t^\star$ for $t=1,\ldots,T$. Second, since ONM uses both first- and second-order information of the loss function $f_t$, the point it generates is closer to $\hat{\bm x}_t$ than that generated by OGD. This suggests that ONM is better at tracking the least-squares estimates than OGD. In fact, these two claims are corroborated by our numerical results; see Figure~\ref{fig:traj_sigma1e-4}. 

In the second scenario, the noise level increases relative to the path variation (i.e., $\sigma_{t+1}=0.01$ vs. $v_t = \tfrac{0.005}{\sqrt{2(t+1)}}$ for $t=1,\ldots,T-1$ with $T=500$). Here, the ability of ONM to track the least-squares estimates closely becomes a liability, because Proposition~\ref{thm:esterror} suggests that the true target position and the least-squares estimate will be further apart. Indeed, as shown in Figure~\ref{fig:error_sigma1e-2}, ONM has a larger CTTE than OGD, and the gap widens as time goes by. We see from Figure~\ref{fig:traj_sigma1e-2} that ONM is much better at tracking the least-squares estimates than OGD. However, the least-squares estimates are quite far from the true target positions, and OGD is better at tracking the latter.

We note that in the above two scenarios, the noise level is constant, and the CTTE of OGD eventually grows linearly (see Figures~\ref{fig:error_sigma1e-4} and~\ref{fig:error_sigma1e-2}). This is consistent with the result in Theorem~\ref{thm:ogd-ctte}, as $N_1(T)=\Theta(T)$ and $N_2(T)=\Theta(T)$ and both terms dominate $V(T)=\Theta(\sqrt{T})$.

In the third scenario, the noise level diminishes as time goes by, but the relative magnitude between noise level and path variation stays roughly constant (i.e., $\sigma_{t+1}=\tfrac{0.01}{\sqrt{t}}$ vs. $v_t = \tfrac{0.005}{\sqrt{2(t+1)}}$ for $t=1,\ldots,T-1$ with $T=500$). From Figure~\ref{fig:error_sigma1e-2oversqrt}, we see that with exact initialization, OGD has a smaller CTTE than ONM. This suggests that the high initial noise level, which causes the least-squares estimate to deviate from the true target position, throws off ONM and degrades its subsequent tracking performace even though the noise level is diminishing. Moreover, given the high initial noise level, the OLS initialization strategy tends to produce an inaccurate estimate of the true initial target position. Consequently, with OLS initialization, the CTTE of both OGD and ONM grow rapidly in the beginning, though the former is more affected by the quality of the OLS estimate than the latter. Nevertheless, we observe that the CTTE gap between OGD and ONM narrows as time goes by. This supports our earlier claim that OGD is better at tracking the true target positions than ONM; see also Figure~\ref{fig:traj_sigma1e-2oversqrt}. Lastly, we note that the CTTE of OGD grows sublinearly. This is consistent with the result in Theorem~\ref{thm:ogd-ctte}, as we have $V(T) = \Theta(\sqrt{T})$, $N_1(T) = \Theta(\sqrt{T})$, and $N_2(T)=\Theta(\log T)$. 

We also compare the per-iteration CPU time of OGD and ONM. As can be seen in Table~\ref{tab:CPUtime}, OGD is about 2-3 times faster than ONM. The higher runtime of the latter can be attributed to the computation of the inverse of the Hessian of the loss function.

\begin{table}[htb]
\fontsize{10}{12.5}\selectfont
	\centering
\begin{tabular}{c|c|c}
Noise Level & OGD & ONM\\
\hline
$\sigma_t=0.0001$ & $5.23\times10^{-6}$s & $1.36\times10^{-5}$s\\
$\sigma_t=0.01$ & $5.11\times10^{-6}$s & $1.31\times10^{-5}$s\\
$\sigma_t=0.01/\sqrt{t}$ & $5.05\times10^{-6}$s & $1.29\times10^{-5}$s
\end{tabular}
\caption{Per-iteration CPU time of OGD and ONM.}
\label{tab:CPUtime}
\end{table}

To better understand the effect of the relative magnitude between noise level and path variation on the tracking performance of OGD and ONM, let us plot Figure~\ref{fig:error_sigma1e-4} again but with the longer time horizon $T=10000$. The result is shown in Figure~\ref{fig:CTTE}. Although the CTTE of OGD is higher than that of ONM in the beginning, the latter eventually overtakes the former as $t$ increases. This is consistent with our earlier observation that ONM is better at tracking the least-squares estimates than OGD. Indeed, when $t$ is sufficiently large, the noise level $\sigma_{t+1}=0.0001$ is larger than the path variation $v_t = \tfrac{0.005}{\sqrt{2(t+1)}}$. Thus, as time goes by, the true target position and the least-squares estimate become further apart (see Proposition~\ref{thm:esterror}), and ONM starts to incur a higher target tracking error at each time step. This suggests that the performance of ONM is rather sensitive to the noise level, while that of OGD is quite stable.

\begin{figure}
    \centering
    \includegraphics[scale=0.48]{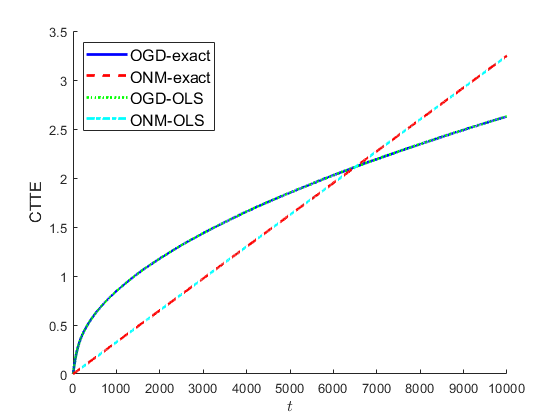}
    \caption{CTTE of OGD and ONM at noise level $\sigma_t=0.0001$, $T=10000$.}
    \label{fig:CTTE}
\end{figure}

As a further illustration, we construct another set of test instances with $T=10000$, the same initial target position $\bm{x}_1^\star = 
\begin{bmatrix}
2 & 1
\end{bmatrix}^T
$ and target trajectory~\eqref{eq:source-update} as before, and the following two different noise levels: (i) $\sigma_t = \frac{0.005}{\sqrt{2t}}$ for $t=1,\ldots,T$; (ii) $\sigma_t = \frac{0.008}{\sqrt{2t}}$ for $t=1,\ldots,T$. For $t=1,\ldots,T-1$, the ratios of noise level $\sigma_{t+1}$ to path variation $v_t$ in these two cases are 1 and 1.6, respectively. Figures~\ref{fig:rand_noise-same-variation}--\ref{fig:rand_noise-1pt6-variation} show the CTTE of OGD and ONM with exact and OLS initialization at these two noise levels.

\begin{figure*}[htb]
\centering
\begin{subfigure}{.42\textwidth}
\centering
\includegraphics[width=0.92\textwidth]{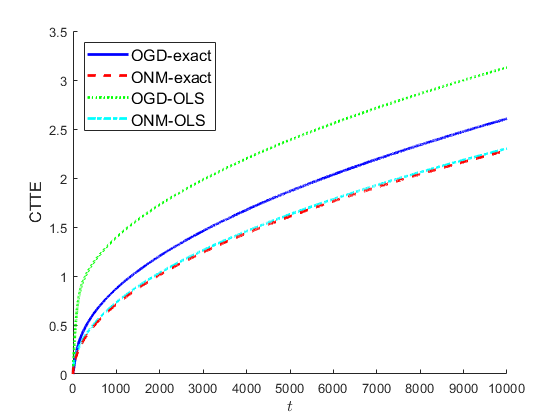}
\caption{$\bm{x}_{t+1}^\star = \bm{x}_t^\star + \frac{0.005}{\sqrt{2(t+1)}} \bm{u}_t,~\sigma_t = \frac{0.005}{\sqrt{2t}}$}
\label{fig:rand_noise-same-variation}
\end{subfigure}
\begin{subfigure}{.42\textwidth}
\centering
\includegraphics[width=0.92\textwidth]{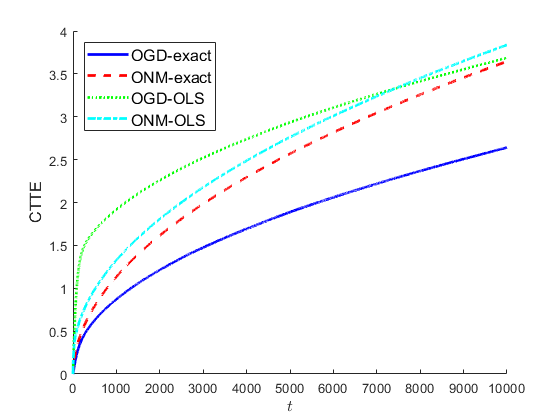}
\caption{$\bm{x}_{t+1}^\star = \bm{x}_t^\star + \frac{0.005}{\sqrt{2(t+1)}} \bm{u}_t,~\sigma_t = \frac{0.008}{\sqrt{2t}}$}
\label{fig:rand_noise-1pt6-variation}
\end{subfigure}
\caption{CTTE of OGD and ONM when applied to different target trajectories and noise levels.}
\label{fig:region_OGD}
\end{figure*}

When the noise level to path variation ratio is 1, Figure~\ref{fig:rand_noise-same-variation} shows that ONM performs better than OGD, regardless of whether exact or OLS initialization is used. However, when the ratio increases to $1.6$, Figure~\ref{fig:rand_noise-1pt6-variation} shows that OGD eventually performs better than ONM, regardless of whether exact or OLS initialization is used. These results corroborate our earlier account that OGD is better at tracking the true target positions, while ONM is better at tracking the least-squares estimates.

\subsection{Large Noise Level and Path Variation}
Next, we study the CTTE of OGD and ONM when the two methods are applied to test instances that violate one or more of the conditions~\eqref{eq:dist-asp-dyn},~\eqref{eq:eps-dyn}, and~\eqref{eq:eps-rad}. In particular, there is no guarantee that the iterate generated by OGD at the current time step lies in the strong convexity region of the loss function at the next time step. 

We first construct a test instance that has large noise level and path variation but the ratio between them is small. The time horizon of interest is set to $T=500$. The target’s initial position is set to $\bm{x}_1^\star = 
\begin{bmatrix}
2 & 1
\end{bmatrix}^T
$
and its subsequent positions are given by
\[
\bm{x}_{t+1}^\star = \bm{x}_t^\star+\frac{0.1}{\sqrt{2(t+1)}}\bm{u}_t,\quad t = 1,\ldots,T-1.
\]
Here, as before, $\bm{u}_1,\ldots,\bm{u}_{T-1} \in \R^2$ are independently and uniformly distributed on the unit circle centered at the origin. The noise levels are given by $\sigma_t=\frac{0.1}{\sqrt{2t}}$ for $t = 1,\ldots,T$. Figure~\ref{fig:rand_noise-var_1e-1oversqrt} shows the CTTE of OGD and ONM. We observe that the CTTE of OGD is much lower than that of ONM with both exact and OLS initialization. One possible explanation is that the good performance of ONM relies heavily on the local strong convexity of the loss function, and the lack of such a property seriously affects its performance.

Now, let us construct a test instance that has a small noise level but large path variation, so that the ratio between them is small. The time horizon of interest and the target's initial position are the same as before. The target trajectory is given by
\[
\bm{x}_{t+1}^\star = \bm{x}_t^\star+\frac{0.5}{\sqrt{2(t+1)}}\bm{u}_t,\quad t = 1,\ldots,T-1,
\]
while the noise levels are given by $\sigma_t=\frac{0.001}{\sqrt{2t}}$ for $t = 1,\ldots,T$. Figure~\ref{fig:rand_noise_1e-3oversqrt-var_5e-1oversqrt} shows the CTTE of OGD and ONM. We see that the CTTE of OGD is much lower than that of ONM. In fact, when the iterates are no longer guaranteed to lie in the strong convexity regions of the loss functions, ONM becomes rather unstable regardless of the noise level to path variation ratio. This supports our earlier explanation that the local strong convexity of the loss function is crucial to the good performance of ONM. By contrast, OGD is much more robust and can better track the true target positions even when the conditions for local strong convexity are violated.

\begin{figure*}[!t]
\centering
\begin{subfigure}{.42\textwidth}
\centering
\includegraphics[width=0.92\textwidth]{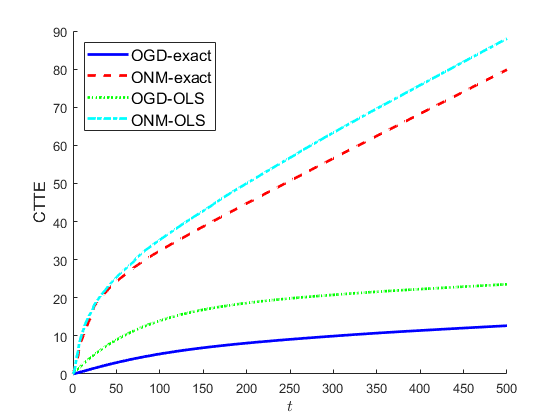}
\caption{$\bm{x}_{t+1}^\star = \bm{x}_t^\star+\frac{0.1}{\sqrt{2(t+1)}}\bm{u}_t,~\sigma_t=\frac{0.1}{\sqrt{2t}}$}
\label{fig:rand_noise-var_1e-1oversqrt}
\end{subfigure}
\begin{subfigure}{.42\textwidth}
\centering
\includegraphics[width=0.92\textwidth]{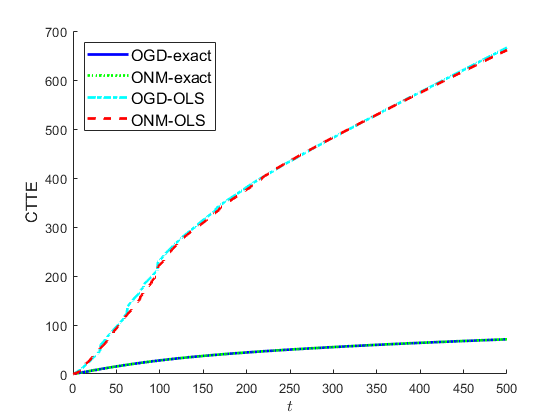}
\caption{$\bm{x}_{t+1}^\star = \bm{x}_t^\star+\frac{0.5}{\sqrt{2(t+1)}}\bm{u}_t,~\sigma_t=\frac{0.001}{\sqrt{2t}}$}
\label{fig:rand_noise_1e-3oversqrt-var_5e-1oversqrt}
\end{subfigure}
\caption{CTTE of OGD and ONM with large noise level and/or path variation.}
\label{fig:large_var&noise}
\end{figure*}

\section{Conclusion} \label{sec:concl}

In this paper, we established the first non-trivial performance bound for OGD when it is applied to a time-varying non-convex least-squares formulation of the TOA-based tracking problem. The performance metric we adopted is the CTTE, which measures the cumulative discrepancy between the trajectory of position estimates and that of the true target. To establish the said performance bound, we developed new results on the estimation and geometric properties of the classic static TOA-based source localization problem, which can be of independent interest. Our numerical results corroborate the theoretical findings and show that OGD can effectively track the target at different noise levels.

A possible future direction is to design and analyze online methods for TDOA-based tracking, which corresponds to a sequential version of the TDOA-based source localization problem (see, e.g.,~\cite{LCS09} and the references therein). One possible approach is to combine the results in~\cite{LPS17} with the techniques developed in this paper. Another future direction is to study the performance of different online methods for solving the TOA-based tracking problem.

\appendix
\subsection{Proof of Proposition~\ref{thm:esterror}} \label{app:esterror}
Following the development in Section~\ref{sec:sim}, the OLS estimate of the true target position $\bm{x}^\star$ is given by
\[ \bm{x}_{\rm OLS} = (\bm{A}^T\bm{A})^{-1}\bm{A}^T\bm{b}, \]
where $\bm{A}$ is given in~\eqref{eq:LS-A} and 
\begin{equation} \label{eq:OLS-b}
\bm{b} := 
\frac{1}{2}\begin{bmatrix}
\|\bm{a}_2\|^2 - \|\bm{a}_1\|^2 + r_1^2 - r_2^2 \\
\vdots\\
\|\bm{a}_m\|^2 - \|\bm{a}_{m-1}\|^2 + r_{m-1}^2 - r_m^2
\end{bmatrix}.
\end{equation}

Now, let $r_i^{\star} = \|\bm{x}^\star - \bm{a}_i\|$ for $i=1,\ldots,m$ and let $\bm{b}^{\star}$ be the vector obtained by replacing $r_i$ with $r_i^{\star}$ in~\eqref{eq:OLS-b}. Then, the derivation in Section~\ref{sec:sim} shows that $\bm{x}^\star$ satisfies $\bm{x}^\star = (\bm{A}^T\bm{A})^{-1}\bm{A}^T\bm{b}^\star$. This implies that
\begin{align*}
\|\bm{x}_{\rm OLS} - \bm{x}^{\star}\| = \|(\bm{A}^T\bm{A})^{-1}\bm{A}^T(\bm{b}-\bm{b}^{\star})\|.
\end{align*}
Using the fact that $r_i = r_i^\star + w_i$ for $i=1,\ldots,m$, we get $r_i^2-(r_i^\star)^2 = 2r_i^\star w_i + w_i^2$ and hence
\[ 
\| \bm{b} - \bm{b}^\star \| \le C_0 \|\bm{w}\| + \frac{1}{2} \|\tilde{\bm w}\|
\] 
for some constant $C_0>0$, where
\[ \tilde{\bm w} := 
\begin{bmatrix}
w_1^2 - w_2^2\\
w_2^2 - w_3^2\\
\vdots\\
w_{m-1}^2 - w_m^2
\end{bmatrix}. \]
Since $\|\tilde{\bm w}\| \le \|\tilde{\bm w}\|_1\le 2\|\bm{w}\|^2$, our assumption on $\|\bm{w}\|$ yields
\begin{align} \label{eq:LS-star}
\|\bm{x}_{\rm OLS} - \bm{x}^{\star}\| \leq C_1\sqrt{m}\sigma + C_2m\sigma^2
\end{align}
for some constants $C_1,C_2>0$.

Next, let $\hat{r}_i = \|\hat{\bm{x}} - \bm{a}_i\|$ for $i=1,\ldots,m$ and let $\hat{\bm{b}}$ be the vector obtained by replacing $r_i$ with $\hat{r}_i$ in~\eqref{eq:OLS-b}. By repeating the same argument as above and noting that
\begin{align*}
r_i^2 - \hat{r}_i^2 &= 2r_i(r_i-\hat{r}_i) - (r_i-\hat{r}_i)^2 \\
&= 2(r_i^\star + w_i)(r_i-\hat{r}_i) - (r_i-\hat{r}_i)^2, \\
f(\hat{\bm x}) &= \sum_{i=1}^m (\hat{r}_i - r_i)^2 \le f(\bm{x}^\star) = \|\bm{w}\|^2 \le c_0^2m\sigma^2, \\
|w_i| &\ll r_i^\star,
\end{align*}
we have
\begin{align*}
\|\bm{b} - \hat{\bm{b}}\| &\le C_3 \sqrt{f(\hat{\bm x})} + \frac{1}{2}\|\tilde{\bm r}\| \le C_3 \sqrt{f(\hat{\bm x})} + C_4 f(\hat{\bm x}) \\
&\le C_5 \sqrt{m}\sigma + C_6 m\sigma^2
\end{align*}
for some constants $C_3,C_4,C_5,C_6>0$, where
\[ \tilde{\bm r} := \begin{bmatrix}
(r_2-\hat{r}_2)^2 - (r_1-\hat{r}_1)^2 \\
(r_3-\hat{r}_3)^2 - (r_2-\hat{r}_2)^2 \\
\vdots\\
(r_m-\hat{r}_m)^2 - (r_{m-1}-\hat{r}_{m-1})^2
\end{bmatrix}.
\]
This gives
\begin{align}\label{eq:LS-hat}
\|{\bm{x}_{\rm OLS}}-\hat{\bm{x}}\| \le C_7 \sqrt{m}\sigma+C_8m\sigma^2 
\end{align}
for some constants $C_7,C_8>0$. The desired result then follows by applying the triangle inequality to~\eqref{eq:LS-star} and~\eqref{eq:LS-hat}. 

\subsection{Proof of Theorem~\ref{thm:str_cvx}} \label{app:str_cvx}
We begin with two technical lemmas.
\begin{lemma}\label{lem:eig-diff-univec}
Let $\bm{u},\bm{v}\in\mathbb{R}^n$ be two linearly independent vectors. Then,
\begin{equation} \label{eq:mineig}
\lambda_{\min}(\bm{u}\bm{u}^T - \bm{v}\bm{v}^T) = \frac{\|\bm{u}\|^2-\|\bm{v}\|^2-\|\bm{u}-\bm{v}\|\|\bm{u}+\bm{v}\|}{2}.
\end{equation}
\end{lemma}
\begin{proof}
Let $(\lambda,\bm{w}) \in \R\times\R^n$ be an eigenpair of $\bm{u}\bm{u}^T-\bm{v}\bm{v}^T$; i.e., $\|\bm{w}\|=1$ and
\[
(\bm{u}\bm{u}^T-\bm{v}\bm{v}^T)\bm{w} = \lambda\bm{w}.
\]
If $\bm{w}\in\mbox{span}\{\bm{u},\bm{v}\}^{\perp}$, then $\lambda = 0$. Otherwise, we can write $\bm{w} = a\bm{u}+b\bm{v}$ for some $a,b\in\R$ and compute
\begin{align}
&~ (\bm{u}\bm{u}^T-\bm{v}\bm{v}^T)(a\bm{u}+b\bm{v}) \nonumber\\
=&~ (a\|\bm{u}\|^2+b(\bm{u}^T\bm{v}))\bm{u} - (a(\bm{u}^T\bm{v})+b\|\bm{v}\|^2)\bm{v} \nonumber\\
=&~ \lambda a\bm{u}+\lambda b\bm{v}. \label{eq:eigpair}
\end{align}
Consider the following cases:

\smallskip
\noindent{\it Case 1}: $\bm{u}^T\bm{v}=0$.

\noindent It is immediate that $\lambda_{\min}(\bm{u}\bm{u}^T - \bm{v}\bm{v}^T) = -\|\bm{v}\|^2$ and the corresponding eigenvector is $\bm{w} = \bm{v}/\|\bm{v}\|$. Since $\|\bm{u}-\bm{v}\|^2 = \|\bm{u}+\bm{v}\|^2 = \|\bm{u}\|^2 + \|\bm{v}\|^2$, we obtain~\eqref{eq:mineig}.

\smallskip
\noindent{\it Case 2}: $\bm{u}^T\bm{v}\not=0$.

\noindent We claim that both $a$ and $b$ must be non-zero. Indeed, suppose to the contrary that $a=0$. Since $\bm{w}\not=\bm{0}$, we have $b\not=0$. It follows from~\eqref{eq:eigpair} that
\[ (\bm{u}^T\bm{v})\bm{u} = (\lambda + \|\bm{v}\|^2)\bm{v}. \]
As $\bm{u}^T\bm{v}\not=0$, we also have $\lambda\not=-\|\bm{v}\|^2$. However, the above relation contradicts the linear independence of $\bm{u}$ and $\bm{v}$. Thus, we conclude that $a\not=0$. A similar argument shows that $b\not=0$.

%
%
%

Now, by equating terms in~\eqref{eq:eigpair}, we have
\begin{align}\label{eq:eig}
\lambda = \frac{a\|\bm{u}\|^2 + b(\bm{u}^T\bm{v})}{a} = \frac{-a(\bm{u}^T\bm{v})- b\|\bm{v}\|^2}{b}.
\end{align}
Solving the quadratic equation
\[ 
(\bm{u}^T\bm{v})a^2+(\|\bm{u}\|^2+\|\bm{v}\|^2)ab +(\bm{u}^T\bm{v})b^2 = 0,
\] 
we obtain a relationship between $a$ and $b$. Plugging this relationship into~\eqref{eq:eig} yields
\begin{align*}
\lambda &= \frac{(\|\bm{u}\|^2+\|\bm{v}\|^2)\pm\sqrt{(\|\bm{u}\|^2+\|\bm{v}\|^2)^2-4(\bm{u}^T\bm{v})^2}}{2} - \|\bm{v}\|^2 \nonumber\\
&= \frac{\|\bm{u}\|^2-\|\bm{v}\|^2\pm\|\bm{u}-\bm{v}\|\|\bm{u}+\bm{v}\|}{2}.
\end{align*}
The negative root gives a non-positive eigenvalue by the Cauchy-Schwarz inequality. This establishes~\eqref{eq:mineig}.
\end{proof}

\begin{lemma}\label{lem:diff_unitvec}
For any $\bm{x},\bm{y}\in\mathbb{R}^n\setminus\{\bm{0}\}$, we have
\[
\left\|\frac{\bm{x}}{\|\bm{x}\|} - \frac{\bm{y}}{\|\bm{y}\|}\right\| \leq\frac{\|\bm{x}-\bm{y}\|}{\min\{\|\bm{x}\|,\|\bm{y}\|\}}.
\]
\end{lemma}
\begin{proof}
If $\|\bm{x}\| = \|\bm{y}\|$, then the inequality trivially holds as equality. Hence, we may assume without loss of generality that $\|\bm{x}\| < \|\bm{y}\|$. Consider
\begin{align*}
\left\|\frac{\bm{x}}{\|\bm{x}\|} - \frac{\bm{y}}{\|\bm{y}\|}\right\|^2 = \frac{1}{\|\bm{x}\|^2}\left\|\bm{x} - \frac{\|\bm{x}\|}{\|\bm{y}\|}\bm{y}\right\|^2.
\end{align*}
Writing $\alpha = \|\bm{x}\|/\|\bm{y}\|$, where $0<\alpha<1$ because $\|\bm{x}\| < \|\bm{y}\|$, we see that
\begin{align*}
&~ \|\bm{x} - \alpha\bm{y}\|^2 \\
=&~ \|\bm{x}\|^2 - 2\bm{x}^T\bm{y} + \|\bm{y}\|^2 + 2(1-\alpha)\bm{x}^T\bm{y} + (\alpha^2 - 1)\|\bm{y}\|^2\\
=&~ \|\bm{x}-\bm{y}\|^2 +(1-\alpha)(2\bm{x}^T\bm{y} - (\alpha+1)\|\bm{y}\|^2)
\end{align*}
and
\begin{align*}
2\bm{x}^T\bm{y} - (\alpha+1)\|\bm{y}\|^2 &\leq 2\|\bm{x}\|\|\bm{y}\| - (\|\bm{x}\|+\|\bm{y}\|)\|\bm{y}\|\\
&= \|\bm{y}\|(\|\bm{x}\|-\|\bm{y}\|)<0.
\end{align*}
This completes the proof.
\end{proof}

Armed with Lemmas~\ref{lem:eig-diff-univec} and~\ref{lem:diff_unitvec}, we are now ready to prove Theorem~\ref{thm:str_cvx}. By Proposition~\ref{thm:esterror} and the assumption of Theorem~\ref{thm:str_cvx}, we have
\begin{equation} \label{eq:hat-anchor}
\| \hat{\bm{x}} - \bm{a}_i \| \ge \| \bm{x}^\star - \bm{a}_i \| - \| \hat{\bm{x}} - \bm{x}^\star \|  > \delta, \quad i=1,\ldots,m. 
\end{equation}
Thus, the loss function $f$ is twice continuously differentiable at $\hat{\bm x}$ with
\begin{align*}
\nabla^2f(\hat{\bm x}) = 2\sum_{i=1}^m\Bigg\{&\frac{r_i}{\|\hat{\bm x}-\bm{a}_i\|^3}(\hat{\bm x} - \bm{a}_i)(\hat{\bm x} - \bm{a}_i)^T \nonumber\\
&+\left(1 - \frac{r_i}{\|\hat{\bm x}-\bm{a}_i\|}\right)\bm{I}\Bigg\}.
\end{align*}
Our goal is to prove that $\nabla^2f(\hat{\bm{x}}+\bm{\epsilon})\succ\bm{0}$ for all $\bm{\epsilon}$ within some ball. In particular, this would imply that $\nabla^2f(\hat{\bm{x}})\succ\bm{0}$. To begin, consider a fixed $i \in \{1,\ldots,m\}$. Since $\Lambda/m \le 1$, we have $\kappa < \delta/10$. This, together with Proposition~\ref{thm:esterror} and the assumption that $\|\bm{\epsilon}\| \le \kappa$, gives
\begin{align*}
\| \hat{\bm x} + \bm{\epsilon} - \bm{a}_i \| &\le \| \hat{\bm x} - \bm{x}^\star \| + \| \bm{x}^\star - \bm{a}_i \| + \| \bm{\epsilon} \| \\
&\le \| \bm{x}^\star - \bm{a}_i \| + K_1\sqrt{m}\sigma + K_2m\sigma^2 + \frac{\delta}{10}.
\end{align*}
Moreover, since $|w_i| \ll \|\bm{x}^\star-\bm{a}_i\|$, we may take $r_i \ge \|\bm{x}^\star-\bm{a}_i\|/2$. Putting these together and using the assumption of Theorem~\ref{thm:str_cvx}, we obtain
\begin{align*}
&~ \frac{r_i}{\| \hat{\bm x} + \bm{\epsilon} - \bm{a}_i \|} \nonumber \\
\ge& ~ \frac{1/2}{ 1 + \|\bm{x}^\star - \bm{a}_i\|^{-1} \left( K_1\sqrt{m}\sigma + K_2m\sigma^2 + \delta/10 \right) }\nonumber \\
\ge& ~ \frac{1}{4}.
\end{align*}
Hence, we can bound
\begin{align}
&~\lambda_{\min}(\nabla^2f(\hat{\bm{x}}+\bm{\epsilon})) \nonumber\\
\noalign{\smallskip}
\geq&~ 2 \cdot \lambda_{\min}\left(\sum_{i=1}^m\frac{r_i}{\|\hat{\bm{x}}+\bm{\epsilon}-\bm{a}_i\|^3}(\hat{\bm{x}}+\bm{\epsilon}-\bm{a}_i)(\hat{\bm{x}}+\bm{\epsilon}-\bm{a}_i)^T\right) \nonumber\\
\noalign{\smallskip}
&~+ 2\sum_{i=1}^m\left(1-\frac{r_i}{\|\hat{\bm{x}}+\bm{\epsilon}-\bm{a}_i\|}\right) \nonumber \\
\noalign{\smallskip}
\ge&~ \frac{1}{2}\cdot \lambda_{\min}\left( \sum_{i=1}^m \left( \left( \frac{\hat{\bm{x}}+\bm{\epsilon}-\bm{a}_i}{\|\hat{\bm{x}}+\bm{\epsilon}-\bm{a}_i\|} \right) \left( \frac{\hat{\bm{x}}+\bm{\epsilon}-\bm{a}_i}{\|\hat{\bm{x}}+\bm{\epsilon}-\bm{a}_i\|} \right)^T \right.\right. \nonumber \\
\noalign{\smallskip}
&\qquad\qquad\qquad\quad - \left.\left. \left( \frac{\bm{x}^\star-\bm{a}_i}{\|\bm{x}^\star-\bm{a}_i\|} \right)\left( \frac{\bm{x}^\star-\bm{a}_i}{\|\bm{x}^\star-\bm{a}_i\|}\right)^T \right) \right) \nonumber \\
\noalign{\smallskip}
&~+ \frac{1}{2} \cdot \lambda_{\min}\left( \sum_{i=1}^m \left( \frac{\bm{x}^\star-\bm{a}_i}{\|\bm{x}^\star-\bm{a}_i\|} \right)\left( \frac{\bm{x}^\star-\bm{a}_i}{\|\bm{x}^\star-\bm{a}_i\|}\right)^T \right) \nonumber \\
\noalign{\smallskip}
&~+ 2\sum_{i=1}^m \frac{\|\hat{\bm{x}}+\bm{\epsilon}-\bm{a}_i\| - r_i}{\|\hat{\bm{x}}+\bm{\epsilon}-\bm{a}_i\|}. \label{eq:minEigenHess}
\end{align}
Now, let us bound the first and last terms in~\eqref{eq:minEigenHess} separately. For the first term, we have
\begin{align}\label{eq:imEigen}
&~\lambda_{\min}\left( \sum_{i=1}^m \left( \left(\frac{\hat{\bm{x}}+\bm{\epsilon}-\bm{a}_i}{\|\hat{\bm{x}}+\bm{\epsilon}-\bm{a}_i\|}\right) \left(\frac{\hat{\bm{x}}+\bm{\epsilon}-\bm{a}_i}{\|\hat{\bm{x}}+\bm{\epsilon}-\bm{a}_i\|}\right)^T \right.\right. \nonumber\\
&\qquad\qquad\quad~ - \left.\left. \left(\frac{\bm{x}^\star-\bm{a}_i}{\|\bm{x}^\star-\bm{a}_i\|}\right)\left(\frac{\bm{x}^\star-\bm{a}_i}{\|\bm{x}^\star-\bm{a}_i\|}\right)^T\right) \right) \nonumber\\
\geq&~ -\frac{1}{2} \sum_{i=1}^m \left\|\frac{\hat{\bm{x}}+\bm{\epsilon}-\bm{a}_i}{\|\hat{\bm{x}}+\bm{\epsilon}-\bm{a}_i\|} - \frac{\bm{x}^\star-\bm{a}_i}{\|\bm{x}^\star-\bm{a}_i\|}\right\|\cdot \nonumber\\
&\qquad\qquad\,\,\, \left\|\frac{\hat{\bm{x}}+\bm{\epsilon}-\bm{a}_i}{\|\hat{\bm{x}}+\bm{\epsilon}-\bm{a}_i\|} + \frac{\bm{x}^\star-\bm{a}_i}{\|\bm{x}^\star-\bm{a}_i\|}\right\| \nonumber\\
\geq&~-\sum_{i=1}^m \left\|\frac{\hat{\bm{x}}+\bm{\epsilon}-\bm{a}_i}{\|\hat{\bm{x}}+\bm{\epsilon}-\bm{a}_i\|} - \frac{\bm{x}^\star-\bm{a}_i}{\|\bm{x}^\star-\bm{a}_i\|}\right\| \nonumber\\
\geq&~ - \sum_{i=1}^m\frac{\|\hat{\bm{x}}+\bm{\epsilon} - \bm{x}^\star\|}{\min\{\|\hat{\bm{x}}+\bm{\epsilon}-\bm{a}_i\|,\|\bm{x}^\star-\bm{a}_i\|\}} \nonumber\\
\geq&~ -\frac{2m}{\delta} \cdot \|\hat{\bm{x}}+\bm{\epsilon} - \bm{x}^\star\|, 
\end{align}
where the first inequality follows from Lemma~\ref{lem:eig-diff-univec}, the second follows by applying triangle inequality to the term $\left\|\frac{\hat{\bm{x}}+\bm{\epsilon}-\bm{a}_i}{\|\hat{\bm{x}}+\bm{\epsilon}-\bm{a}_i\|} + \frac{\bm{x}^\star-\bm{a}_i}{\|\bm{x}^\star-\bm{a}_i\|}\right\|$, the third follows from Lemma~\ref{lem:diff_unitvec}, and the last is due to
\[ \min\{\| \hat{\bm{x}}+\bm{\epsilon}-\bm{a}_i\|,\|\bm{x}^\star-\bm{a}_i\|\} \ge \frac{\delta}{2}, \]
which follows from~\eqref{eq:hat-anchor} and the assumption that $\|\bm{\epsilon}\| \le \kappa < \delta/2$. For the last term in~\eqref{eq:minEigenHess}, we use 
\[ \big| \|\hat{\bm{x}}+\bm{\epsilon}-\bm{a}_i\| - r_i \big| \le \| \hat{\bm x} + \bm{\epsilon} - \bm{x}^\star \| + |w_i|, \]
\eqref{eq:hat-anchor}, and the assumption that $\|\bm{\epsilon}\| \le \kappa < \delta/2$ to bound
\begin{align} 
\sum_{i=1}^m \frac{\|\hat{\bm{x}}+\bm{\epsilon}-\bm{a}_i\| - r_i}{\|\hat{\bm{x}}+\bm{\epsilon}-\bm{a}_i\|} &\ge -\sum_{i=1}^m \frac{|w_i| + \| \hat{\bm x} + \bm{\epsilon} - \bm{x}^\star \|}{ \| \hat{\bm{x}}+\bm{\epsilon}-\bm{a}_i\| } \nonumber \\
&\ge -\frac{2}{\delta} \sum_{i=1}^m \left( |w_i| + \| \hat{\bm x} + \bm{\epsilon} - \bm{x}^\star \| \right). \label{eq:t3}
\end{align}
Upon substituting~\eqref{eq:imEigen} and~\eqref{eq:t3} into~\eqref{eq:minEigenHess} and noting that $\sum_{i=1}^m|w_i| \le \sqrt{m}\|\bm{w}\| \le c_0m\sigma$ by assumption, we obtain
\begin{align}\label{eq:minEigenHess-final}
&~ \lambda_{\min}(\nabla^2f(\hat{\bm{x}}+\bm{\epsilon}))\nonumber\\
\geq&~\frac{1}{2} \cdot \lambda_{\min}\left( \sum_{i=1}^m \left( \frac{\bm{x}^\star-\bm{a}_i}{\|\bm{x}^\star-\bm{a}_i\|} \right)\left( \frac{\bm{x}^\star-\bm{a}_i}{\|\bm{x}^\star-\bm{a}_i\|}\right)^T \right) \nonumber \\
&~- \frac{5m}{\delta} (\| \hat{\bm x} - \bm{x}^\star \| + \|\bm{\epsilon}\|) - \frac{4c_0m\sigma}{\delta}.
\end{align}
Using Proposition~\ref{thm:esterror}, we see that the right-hand side of~\eqref{eq:minEigenHess-final} is positive whenever
\begin{align}
\|\bm{\epsilon}\| &< \frac{\delta}{10m} \cdot \lambda_{\min}\left( \sum_{i=1}^m \left( \frac{\bm{x}^\star-\bm{a}_i}{\|\bm{x}^\star-\bm{a}_i\|} \right)\left( \frac{\bm{x}^\star-\bm{a}_i}{\|\bm{x}^\star-\bm{a}_i\|}\right)^T \right) \nonumber\\
&\quad\,\,\, - (K_1\sqrt{m}\sigma + K_2m\sigma^2) - \frac{4c_0\sigma}{5}. \label{eq:ball-rad}
\end{align}
Since $\sigma$ satisfies~\eqref{eq:eps}, the right-hand side of~\eqref{eq:ball-rad} is positive. This completes the proof.

\bibliographystyle{IEEEtran}
\bibliography{ref}

\begin{IEEEbiography}[{\includegraphics[width=1in,height=1.25in,clip,keepaspectratio]{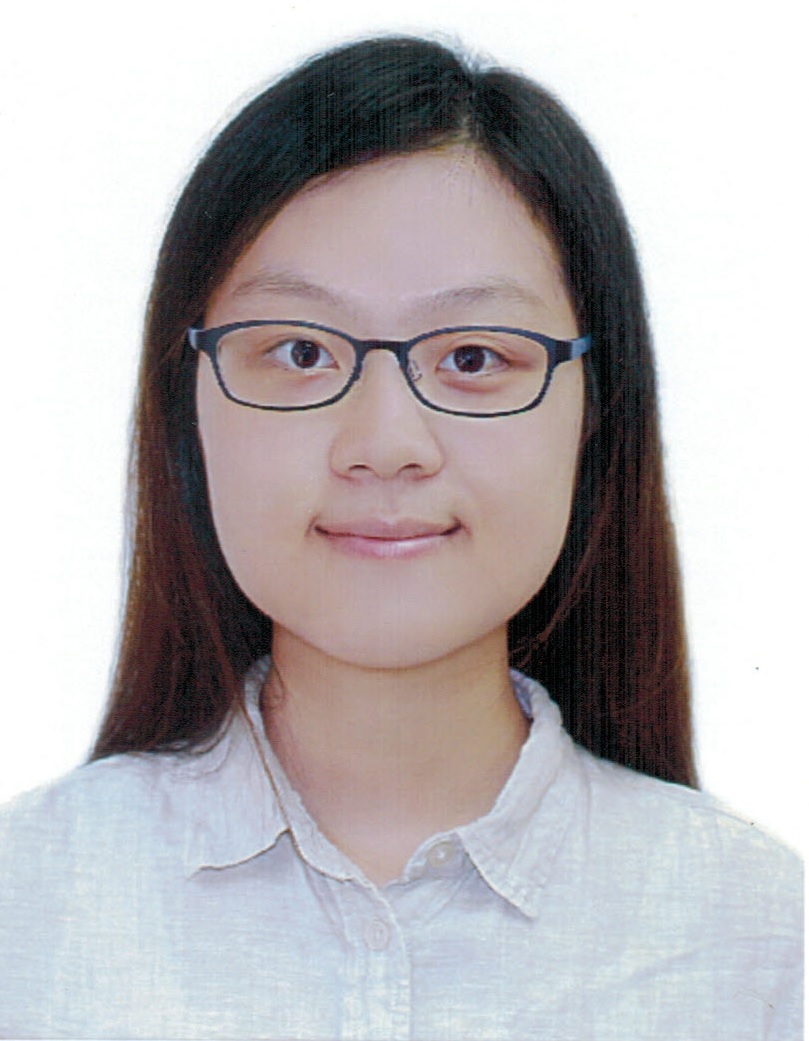}}]
{Yuen-Man Pun} received the BSc degree in Mathematics and MPhil degree in Systems Engineering and Engineering Management (SEEM), both from The Chinese University of Hong Kong (CUHK). She is currently pursuing a PhD degree in SEEM at CUHK. Her research focuses on algorithmic design and analysis and its applications in data science and signal processing.
\end{IEEEbiography}

\begin{IEEEbiography}[{\includegraphics[width=1in,height=1.25in,clip,keepaspectratio]{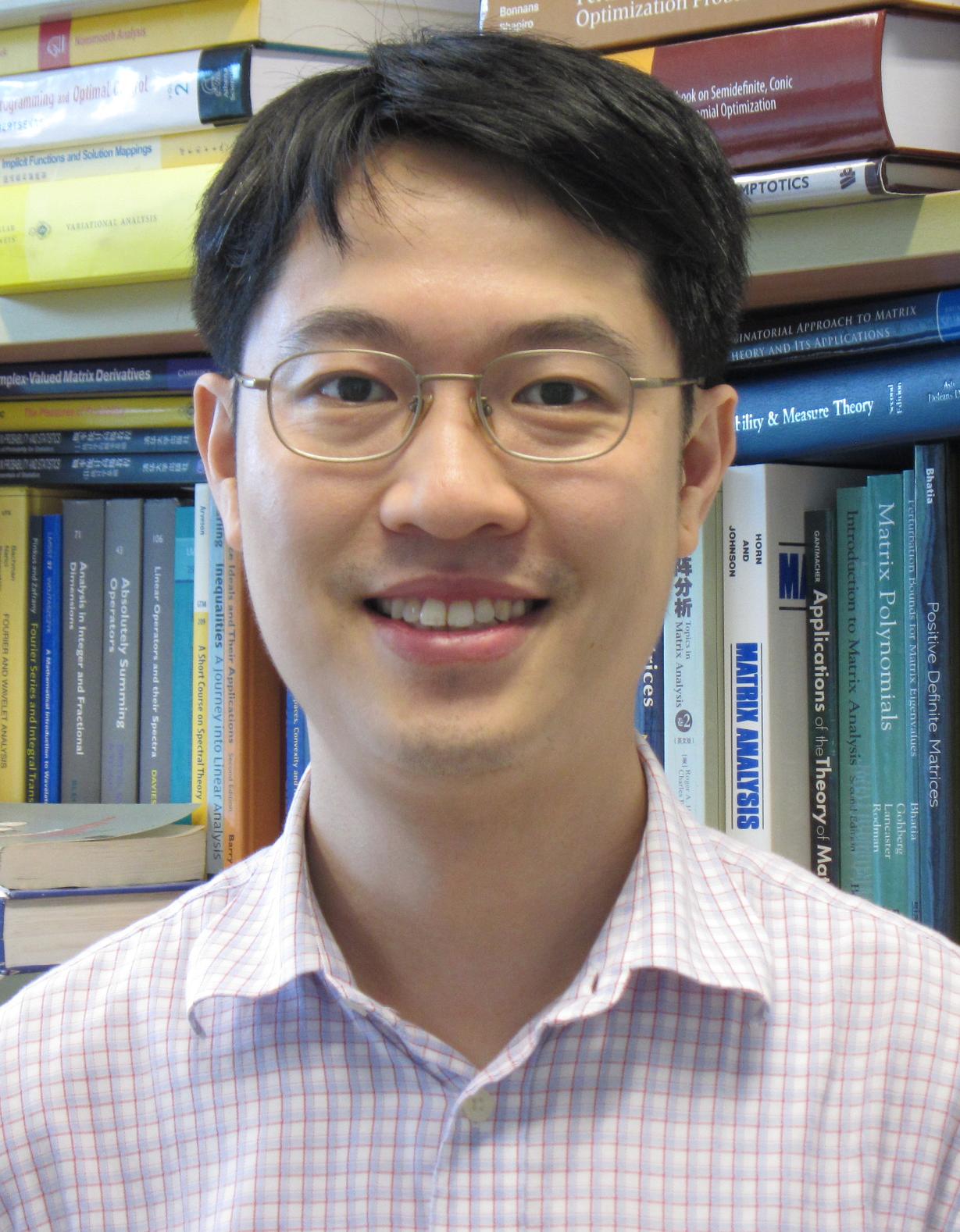}}]
{Anthony Man-Cho So} (M'12-SM'17) received the BSE degree in Computer Science from Princeton University, Princeton, NJ, USA, with minors in Applied and Computational Mathematics, Engineering and Management Systems, and German Language and Culture. He then received the M.Sc. degree in Computer Science and the Ph.D. degree in Computer Science with a Ph.D. minor in Mathematics from Stanford University, Stanford, CA, USA.
		
	Dr. So joined The Chinese University of Hong Kong (CUHK) in 2007. He is now the Associate Dean of Student Affairs in the Faculty of Engineering, Deputy Master of Morningside College, and Professor in the Department of Systems Engineering and Engineering Management. His research focuses on optimization theory and its applications in various areas of science and engineering, including computational geometry, machine learning, signal processing, and statistics.
	
	Dr. So is appointed as an Outstanding Fellow of the Faculty of Engineering at CUHK in 2019. He has received a number of research and teaching awards, including the 2018 IEEE Signal Processing Society Best Paper Award, the 2015 IEEE Signal Processing Society Signal Processing Magazine Best Paper Award, the 2014 IEEE Communications Society Asia-Pacific Outstanding Paper Award, the 2013 CUHK Vice-Chancellor's Exemplary Teaching Award, and the 2010 Institute for Operations Research and the Management Sciences (INFORMS) Optimization Society Optimization Prize for Young Researchers. He currently serves on the editorial boards of Journal of Global Optimization, Optimization Methods and Software, and SIAM Journal on Optimization. He was also the Lead Guest Editor of {\sc IEEE Signal Processing Magazine} Special Issue on ``Non-Convex Optimization for Signal Processing and Machine Learning''.
\end{IEEEbiography}

%
\addtolength{\textheight}{-3cm}   
\end{document}